\documentclass[12pt, a4paper]{amsart}
\usepackage[T1]{fontenc}

\usepackage{amsmath}
\usepackage{amsfonts}
\usepackage{amssymb}
\usepackage{amsthm}
\usepackage{url}

\newcommand{\R}{\mathbb{R}}
\newcommand{\B}{\mathcal{B}}
\newcommand{\dif}[0]{\ensuremath{\,\mathrm{d}}}
\newcommand{\norm}[1]{\ensuremath{\Vert #1 \Vert}}

\newcommand{\abs}[1]{\ensuremath{\vert #1 \vert}}
\newcommand{\scabs}[1]{\ensuremath{\left\vert #1 \right\vert}}

\newcommand{\UpperClass}[0]{\ensuremath{\mathfrak{U}}}
\newcommand{\LowerClass}[0]{\ensuremath{\mathfrak{L}}}
\newcommand{\UpperPerron}[0]{\ensuremath{\overline{H}}}
\newcommand{\LowerPerron}[0]{\ensuremath{\underline{H}}}
\newcommand{\Perron}[0]{\ensuremath{H}}

\DeclareMathOperator*{\osc}{osc}

\def\XXint#1#2#3{{\setbox0=\hbox{$#1{#2#3}{\int}$}
     \vcenter{\hbox{$#2#3$}}\kern-.5\wd0}}

\theoremstyle{plain}
\newtheorem{theorem}[equation]{Theorem}
\newtheorem{lemma}[equation]{Lemma}
\newtheorem{proposition}[equation]{Proposition}

\theoremstyle{remark}
\newtheorem{remark}[equation]{Remark}

\numberwithin{equation}{section}

\theoremstyle{definition}
\newtheorem{definition}[equation]{Definition}

\begin{document}

\title{Perron's method for the porous medium equation}
\author{Juha Kinnunen, Peter Lindqvist \and Teemu Lukkari }
\address[Juha Kinnunen]{Department of Mathematics, Aalto University, P.O. Box 11100, 
FI-00076 Aalto University, Finland}
\email{juha.k.kinnunen@aalto.fi}

\address[Peter Lindqvist]{Department of Mathematics,Norwegian University of Science and Technology,
N-7491 Trondheim, Norway}
\email{lqvist@math.ntnu.no}

\address[Teemu Lukkari]{Department of Mathematics and Statistics, P.O. Box 35 (MaD),
FI-40014 University of Jyv\"askyl\"a, Finland}
\email{teemu.j.lukkari@jyu.fi}

\thanks{The research is supported by the Academy of Finland. Part of
  this paper was written during the authors' stay at the Institut
  Mittag-Leffler in Djursholm. }

\begin{abstract}
This work extends Perron's method for the porous medium equation in the slow diffusion case. 
The main result shows that nonnegative continuous boundary functions are resolutive
in a general cylindrical domain.
\end{abstract}

\maketitle

\section{Introduction}

The porous medium equation
\[
\frac{\partial u}{\partial t}-\Delta u^m=0
\]
is an important prototype of a nonlinear parabolic equation and it is
by now well understood. See the monographs \cite{DK},
\cite{VazquezBook} and \cite{Kiinalaiset} for more on this topic.
However, little is known about the boundary behaviour of solutions in
irregular domains and with general boundary values, except for the
case $m=1$, when we have the classical heat equation~\cite{Watson}.
We shall consider this challenging question. Our main objective is to
apply the method, introduced for harmonic
functions by Perron \cite{Perron}, to this fascinating nonlinear equation. We focus on the slow
diffusion case $m>1$ in cylindrical domains.  For simplicity, we only
consider nonnegative and bounded boundary functions, in which case the
solutions are nonnegative and bounded as well, by the comparison
principle.  However, it is of utmost importance to allow solutions to
attain the value zero, so that moving boundaries, such as those
exhibited by the Barenblatt solution, are not excluded.

We consider the boundary value problem
\begin{displaymath}
  \begin{cases}
    \dfrac{\partial u}{\partial t}-\Delta u^m=0\quad\text{in}\quad\Omega_T,&\\
    u=g\quad\text{on}\quad\partial\Omega\times[0,T), & \\
    u(x,0)=g(x,0),
  \end{cases}
\end{displaymath}
in a bounded open space-time cylinder $\Omega_T=\Omega\times(0,T)$ in
$\R^{N+1}$.  The precise definitions of the solution and the boundary
conditions will be given later. For a given boundary value function
$g$, Perron's method produces two functions: the upper solution
$\UpperPerron_g$ and the lower solution $\LowerPerron_g$ with
$\LowerPerron_g\leq \UpperPerron_g$. Our first main result is that the
upper and lower Perron solutions are indeed weak solutions of the
porous medium equation.  However, the upper and lower solutions may
still take the wrong boundary values. The construction can be
performed not only for space-time cylinders but also for more general
domains in $\R^{N+1}$.

A central question in this theory is to determine when the upper and
lower solutions are the same function. A classical result in this
direction is \emph{Wiener's resolutivity theorem} for harmonic
functions: if the boundary value function is continuous, the upper and
lower Perron solutions coincide, see \cite{WienerResol}.  Our second
main result extends this to the porous medium equation. 
More precisely, nonnegative continuous boundary functions are resolutive for the
porous medium equation in general cylindrical domains in the slow diffusion case $m>1$. No regularity
assumptions on the base of the space-time cylinder are needed.  As far
as we know, the corresponding result for more general domains in
$\R^{N+1}$ remains open.

Perron's method requires a parabolic comparison principle so that the
upper and lower Perron solutions can be defined consistently. Our
first step is to establish a comparison principle in general space-time
cylinders. To prove the resolutivity theorem we first reduce the
situation to smooth boundary values by approximation.  The key step in
the proof for smooth boundary values is constructing super- and
subsolutions which are sufficiently regular in the time direction.  We
use a penalized problem related to the obstacle problem for the porous
medium equation for this purpose, see \cite{BCL}.  Delicate
approximation results and energy estimates play a pivotal role in the
argument.  We hope that these results will have other applications as
well. It is likely that our results and methods also apply to more general equations 
of the type
\[
\frac{\partial u}{\partial t}-\Delta A(u)=0,
\]
see \cite{DK} and \cite{VazquezBook}.

\section{Weak solutions and weak supersolutions}

In this section, we discuss notion on which the construction of
Perron solutions will be based. First, we introduce some notation.

Let $\Omega$ be an open and bounded subset of $\R^N$, and let
$0<t_1<t_2<T$. We denote space-time cylinders by
$\Omega_T=\Omega\times(0,T)$ and $U_{t_1,t_2}=U\times (t_1,t_2)$,
where $U\subset\Omega$ is an open set. We call a cylinder
$U_{t_1,t_2}$ \emph{regular} if the boundary of the base set $U$ is
smooth. The \emph{parabolic boundary} of a space-time cylinder
$U_{t_1,t_2}$ is the set
\begin{displaymath}
  \partial_p U_{t_1,t_2}=(\overline{U}\times\{t_1\})
  \cup (\partial U\times [t_1,t_2]),
\end{displaymath}
i.e. only the initial and lateral boundaries are taken into account.

We use the notation $H^1(\Omega)$ for the Sobolev space consisting of
functions $u$ in $L^2(\Omega)$ such that the weak gradient exists and
also belongs to $L^2(\Omega)$. The Sobolev space with zero boundary
values $H^1_0(\Omega)$ is the completion of $C^{\infty}_0(\Omega)$ in
$H^1(\Omega)$.  The parabolic Sobolev space $L^2(0,T;H^1(\Omega))$
consists of measurable functions $u:\Omega_T\to[-\infty,\infty]$ such
that $x\mapsto u(x,t)$ belongs to $H^1(\Omega)$ for almost all
$t\in(0,T)$, and
\[
\iint_{\Omega_T}\left(\abs{u}^2+\abs{\nabla u}^2\right)\dif x\dif t<\infty.
\]
The definition of the space $L^2(0,T;H^{1}_0(\Omega))$ is similar.  We
say that $u\in L^2_{loc}(0,T;H^{1}_{loc}(\Omega))$ if $u$ belongs to
the parabolic Sobolev space for all $U_{t_1,t_2}\Subset\Omega_T$.
The symbol $\Subset$ means that the set is compactly containd in the bigger set.

\begin{definition}
  Assume that $m>1$.  A nonnegative function $u:\Omega_T\to\R$ is a
  \emph{weak solution} of the porous medium equation
  \begin{equation}\label{eq:pme}
    \frac{\partial u}{\partial t}-\Delta u^m=0
  \end{equation}
  in $\Omega_T$, if $u^m\in  L^2_{loc}(0,T;H^{1}_{loc}(\Omega))$ and
  \begin{equation}\label{eq:weak-pme}
    \iint_{\Omega_T}\left(-u\frac{\partial\varphi}{\partial t}
    +\nabla u^m\cdot\nabla\varphi\right)\dif x\dif t=0
  \end{equation}
  for all smooth test functions $\varphi$ compactly supported in
  $\Omega_T$.  We define \emph{weak supersolutions} by requiring that the
  integral in \eqref{eq:weak-pme} is nonnegative for nonnegative test
  functions $\varphi$.
\end{definition}
Throughout the work we assume that $m>1$. 
It is an interesting question whether corresponding results can be proved also when $m<1$.
Our results and methods also apply to solutions with a changing sign, but we have chosen to consider only nonnegative solutions for simplicity.
However, it is important to allow solutions to
attain the value zero.

Weak solutions are locally H\"older continuous
after a possible redefinition on a set of $(N+1)$-dimensional measure
zero; see \cite{DahlbergKenig, DiBenedettoFriedman, Kiinalaiset} or \cite[Chapter 7]{VazquezBook}.  Thus, without loss of generality, we
may assume that solutions are continuous.  Moreover, weak
supersolutions are lower semicontinuous after a redefinition on a set
of $(N+1)$-dimensional measure zero, see~\cite{AL}.

Besides a local notion of weak solutions, we need a concept of weak
solutions to the initial-boundary value problem
\begin{displaymath}
  \begin{cases}
    \dfrac{\partial u}{\partial t}-\Delta u^m=0\quad\text{in}\quad\Omega_T,&\\
    u=g\quad\text{on}\quad\partial\Omega\times[0,T), & \\
    u(x,0)=g(x,0),
  \end{cases}
\end{displaymath}
where $g$ is a positive, continuous function defined on
$\overline{\Omega}_T$ with $g^m\in L^2(0,T;H^1(\Omega))$.  The lateral
boundary condition is interpreted in the Sobolev sense, meaning that
$u^m-g^m\in L^2(0,T;H^1_0(\Omega))$. The initial condition is
incorporated into the weak formulation by requiring that
\begin{equation}
    \label{eq:ibvp}
      \iint_{\Omega_T}\left(-u\frac{\partial\varphi}{\partial t}
        +\nabla u^m\cdot\nabla\varphi\right)\dif x\dif t=\int_\Omega
      g(x,0)\varphi(x,0)\dif x
\end{equation}
for smooth test functions $\varphi$ vanishing at the time $T$ and with
compact support in space.  With this definition, solutions to the
initial-boundary value problem are unique.  This follows by an
application of Ole\u\i nik's test function, see the proof of
Lemma~\ref{lem:epsilon-conv} below. It is straightforward to check
that the initial values are attained in this sense if and only if
\begin{equation}\label{eq:initial-vals-dist}
  \lim_{t\to 0}\int_{\Omega}u(x,t)\eta(x)\dif
  x=\int_{\Omega}g(x,0)\eta(x)\dif x
\end{equation}
for all $\eta\in C_0^\infty(\Omega)$.

\section{Viscosity supersolutions}

We will employ the notion of \emph{viscosity supersolutions} to
\eqref{eq:pme}, following \cite{KinnunenLindqvist2}.  The term
``viscosity'' is used here purely as a label. In the case $m=1$ this
definition gives supertemperatures, see \cite{Watson}.
For the more common definition of viscosity solutions using pointwise touching test functions, we refer to \cite{BrandleVazquez} and \cite{CV}. 
It is an interesting question whether the two definitions give the same class of functions.

\begin{definition}\label{def:viscosity-supersols}
  A function $u:\Omega_T\to [0,\infty]$ is a \emph{viscosity
  supersolution}, if
  \begin{enumerate}
  \item $u$ is lower semicontinuous,
  \item $u$ is finite in a dense subset of $\Omega_T$, and
  \item the following comparison principle holds: Let
    $U_{t_1,t_2}\Subset\Omega$, and let $h$ be a solution to
    \eqref{eq:pme} which is continuous in $\overline{U_{t_1,t_2}}$.
    If $h\leq u$ on $\partial_p U_{t_1,t_2}$, the $h\leq u$  in
    $U_{t_1,t_2}$.
  \end{enumerate}

  The definition of \emph{viscosity subsolutions} is similar; they are
  upper semicontinuous, and the inequalities in the comparison
  principle are reversed.
\end{definition}
Observe that these functions are defined at \emph{every} point.  A
similar definition was introduced by F. Riesz \cite{Riesz} for the
Laplacian.
The fundamental example of a viscosity supersolution in the sense of
Definition \ref{def:viscosity-supersols} is the Barenblatt solution
\cite{Barenblatt, ZeldovichKompaneets}, which is given by the formula
\begin{displaymath}
  \B_m(x,t)=
  \begin{cases}
    t^{-\lambda}\left(C-\frac{\lambda(m-1)}{2mn}
      \frac{\abs{x}^2}{t^{2\lambda/n}}\right)_+^{1/(m-1)}, & t>0,\\
    0, & t\leq 0,
  \end{cases}
\end{displaymath}
where
\begin{displaymath}
  \lambda=\frac{n}{n(m-1)+2}.
\end{displaymath}
The constant $C$ is usually chosen so that
\begin{displaymath}
  \int_{\Omega}\B_m(x,t)\dif x=1
\end{displaymath}
for all $t>0$.  It is a viscosity supersolution, but not a weak
supersolution. This is due to the lack of integrability of the
gradient.  For other examples, see \cite{VazquezBook}.  However,
\emph{bounded} viscosity supersolutions are weak supersolutions. In
particular, their $m$th power belongs to
$L^2_{loc}(0,T;H^1_{loc}(\Omega))$. Weak supersolutions are viscosity
supersolutions, provided that they are lower semicontinuous, see
\cite{AL} and \cite{KinnunenLindqvist2}. In the present work, it is
enough for the reader  to consider lower semicontinuous weak supersolutions
that are defined at each point in their domain.

Some properties are immediate consequences of the definition.  The
pointwise minimum of a finite number of viscosity supersolutions is a
viscosity supersolution. In particular, the truncations $\min\{u,k\}$,
$k=1,2,\dots$, of a viscosity supersolution $u$ are viscosity
supersolutions. The fact that an increasing limit of viscosity
supersolutions is a viscosity supersolution, provided that the limit
is finite in a dense subset, also follows directly from the
definition.

Our main interest is Perron solutions with continuous boundary values
in irregular domains. In this context, the situation does not change
if one only considers bounded viscosity super- and subsolutions, and
we shall do so from now on.

We begin with the definition of the \emph{Poisson modification} of
a viscosity supersolution.  Let $u$ be a bounded viscosity
supersolution  and $U_{t_1,t_2}\Subset\Omega_T$ be a regular space-time
cylinder.  We define
\begin{displaymath}
  P(u,U_{t_1,t_2})=
  \begin{cases}
    u\quad\text{in}\quad\Omega_T\setminus U_{t_1,t_2},&\\
    h\quad\text{in}\quad U_{t_1,t_2},&
  \end{cases}
\end{displaymath}
where $h$ is the solution in $U_{t_1,t_2}$ with boundary values $u$.
The function $h$ is constructed as follows: by semicontinuity, we find
an increasing sequence $(\varphi_k)$ of smooth functions  converging to
$u$ pointwise in $U_{t_1,t_2}$ as $k\to\infty$. Let $h_k$ be the
solution with values $\varphi_k$ on the parabolic boundary of
$U_{t_1,t_2}$. Then $h_k\leq u$ and the sequence $(h_k)$ is increasing
by the comparison principle. It follows that $h=\lim h_k$ is a
solution in $U_{t_1,t_2}$. Further, it is easy to verify that
$P(u,U_{t_1,t_2})$ is a viscosity supersolution in $\Omega_T$, see
\cite[pp. 157--158]{KinnunenLindqvist2}.

We need an auxiliary result to bypass the fact that we may not add
constants to solutions.
\begin{lemma}\label{lem:epsilon-conv}
  Assume that $g\in C(\overline{\Omega}_T)$ is a function such that
  $g^m\in L^2(0,T;H^1(\Omega))$ and $0\leq g\leq M$. Define the
  function $g_\varepsilon$ by
  \begin{displaymath}
    g_\varepsilon=(g^m+\varepsilon^m)^{1/m},
  \end{displaymath}
  where $0\leq \varepsilon\leq 1$.  Let $u$ and $u_\varepsilon$ be the
  unique weak solutions in the sense of \eqref{eq:ibvp} to the initial-boundary value problems
  \begin{displaymath}
    \begin{cases}
     \dfrac{\partial u}{\partial t}-\Delta u^m=0\quad\text{in}\quad\Omega_T, &\\
      u^m-g^m\in L^2(0,T;H^1_0(\Omega)), &\\
      u(x,0)=g(x,0),
    \end{cases}
  \end{displaymath}
  and 
  \begin{displaymath}
    \begin{cases}
     \dfrac{\partial u_\varepsilon}{\partial t}-\Delta u^m_{\varepsilon}=0 \quad\text{in}\quad\Omega_T,&\\
      u^m_\varepsilon-g^m_\varepsilon\in L^2(0,T;H^1_0(\Omega)), &\\
      u_\varepsilon(x,0)=g_\varepsilon(x,0),
    \end{cases}
  \end{displaymath}
  respectively.
 Then we have
  \begin{displaymath}
    \iint_{\Omega_T}(u_\varepsilon-u)(u^m_{\varepsilon}-u^m)\dif
    x\dif t\leq \varepsilon^m \abs{\Omega_T}(M+1)+\varepsilon 
    \abs{\Omega_T}(M+1)^m.
  \end{displaymath}
\end{lemma}
\begin{proof}
  We use the so-called Ole\u\i nik's test function. The function
  $u_\varepsilon^m-u^m-\varepsilon^m$ has zero boundary values on the
  lateral boundary in Sobolev's sense. Thus
  \begin{displaymath}
    \eta(x,t)=
    \begin{cases}
      \int_t^T(u_\varepsilon^m-u^m-\varepsilon^m)\dif s, & 0<t<T,\\
      0, & t\geq T,
    \end{cases}
  \end{displaymath}
  is an admissible test function for the equations satisfied by $u$
  and $u_\varepsilon$.  This gives
  \[
   \iint_{\Omega_T}\left(-u\frac{\partial \eta}{\partial t}+\nabla u^m\cdot
    \nabla \eta\right) \dif x \dif t=\int_{\Omega}u(x,0)\eta(x,0)\dif x,
    \]
   and
    \[ 
    \iint_{\Omega_T}\left(-u_\varepsilon\frac{\partial \eta}{\partial t}+\nabla u^m_\varepsilon\cdot
    \nabla \eta\right) \dif x \dif t=\int_{\Omega}u_\varepsilon(x,0)\eta(x,0)\dif x.
  \]
  Since we have
  \[
    \eta_t=-(u_\varepsilon^m-u^m)+\varepsilon^m
    \quad\text{and}\quad
    \nabla \eta=\int_t^T\nabla (u^m_\varepsilon -u^m)\dif s,
  \]
  we obtain 
  \begin{align*}
    \iint_{\Omega_T}&\bigg((u_\varepsilon -u)(u^m_\varepsilon-u^m-\varepsilon^m)\\
    &\qquad+\nabla
    (u^m_\varepsilon-u^m)\cdot\int_t^T\nabla(u_\varepsilon^m-u^m)\dif s\bigg)\dif
    x\dif t\\
    =&\int_{\Omega}(g_\varepsilon(x,0)-g(x,0))\left(\int_0^T(u_\varepsilon^m-u^m-\varepsilon^m)\dif
    s\right)\dif x
  \end{align*}
  by subtracting the equations. Integration with respect to the variable $t$ shows
  that the triple integral equals to
  \begin{displaymath}
    \frac{1}{2}\int_\Omega\left(\int_0^T(\nabla u^m_\varepsilon-\nabla
      u^m)\dif s\right)^2\dif x,
  \end{displaymath}
  which is a positive quantity. Thus we get the estimate
  \begin{displaymath}
    \begin{aligned}
    \iint_{\Omega_T}&(u-u_\varepsilon)(u^m-u_\varepsilon^m)\dif x\dif t
    \leq \varepsilon^m \iint_{\Omega_T}(u_\varepsilon-u)\dif x\dif t\\
    &+\int_{\Omega}(g_\varepsilon(x,0)-g(x,0))
   \left( \int_0^T(u_\varepsilon^m-u^m)\dif
    s\right)\dif x\\
    &-\varepsilon^m T \int_{\Omega}(g_\varepsilon(x,0)-g(x,0))\dif x.
    \end{aligned}
  \end{displaymath}
  The last term on the right-hand side is negative, since $g_\varepsilon\geq
  g$, and we simply discard it. Furthermore,   by the definition of $g_\varepsilon$,  we have
  \begin{displaymath}
    g_\varepsilon-g=(g^m+\varepsilon^m)^{1/m}-g\leq \varepsilon
  \end{displaymath}
and, by the maximum principle, we conclude that $u\leq M$ and $u_\varepsilon\leq M+1$.
   The required estimate  follows, since
  \begin{displaymath}
    u_\varepsilon-u\leq M+1 \quad\text{and}\quad
    u_\varepsilon^m-u^m\leq (M+1)^m.\qedhere
  \end{displaymath}  
\end{proof}

We conclude by a comparison principle between viscosity sub- and
supersolutions. The essential feature here is that the base $\Omega$
of the space-time cylinder $\Omega_T$ may be an arbitrary bounded open
set. For the comparison principle in regular cylinders, 
see \cite[pp.~10-12]{DK} or  \cite[pp.~132-134]{VazquezBook}.

\begin{theorem}[Comparison Principle]\label{thm:comparison}
  Let $u$ be a bounded viscosity subsolution and $v$ a bounded
  viscosity supersolution such that
  \begin{equation}\label{eq:comp}
    \limsup_{z\to z_0}u\leq \liminf_{z\to z_0} v
  \end{equation}
  for all $z_0\in \partial_p \Omega_T$. Then
    $u\leq v$ in $\Omega_T$.
\end{theorem}
\begin{proof}
  Let $\varepsilon_j=1/j$, $j=1,2,3,\ldots$.  By \eqref{eq:comp}, we
  can find regular cylinders $Q_j=U_j\times (t_j,T)$,
 with  $U_j\Subset\Omega$, such that
  \begin{displaymath}
    u^m\leq v^m+\varepsilon_j^m \quad\text{in}\quad \Omega_T\setminus Q_j.
  \end{displaymath}
  
  Let $w_j$ be the weak solution in $Q_j$ with boundary values given by $v$
  on $\partial_p Q_j$, and let $\widetilde{w}_j$ be the weak solution with
  boundary values $(v^m+\varepsilon_j^m)^{1/m}$ on $\partial_p Q_j$ in the sense of \eqref{eq:ibvp}. 
Define the
  functions $h_j$ and $\widetilde{h}_j$ by
  \begin{displaymath}
    h_j=
    \begin{cases}
      v& \quad\text{in}\quad\Omega_T\setminus Q_j,\\
      w_j&\quad\text{in}\quad Q_j,
    \end{cases}
 \end{displaymath}
     and
  \begin{displaymath}
      \widetilde{h}_j=
    \begin{cases}
      (v^m+\varepsilon^m_j)^{1/m}\quad\text{in}\quad\Omega_T\setminus Q_j,&
      \\
      \widetilde{w}_j\quad\text{in}\quad Q_j.&
    \end{cases}
  \end{displaymath}
  Recall that $Q_j$ is a regular cylinder. An application of the
  comparison principle on $Q_j$ shows that
  \begin{displaymath}
    h_j\leq v \quad \text{and}\quad u\leq \widetilde{h}_j \quad \text{in }\Omega_T.
  \end{displaymath}
  Now
  \begin{displaymath}
    0\leq (u-v)_+(u^m-v^m)_+\leq 
    \begin{cases}
      (\widetilde{h}_j-h_j)(\widetilde{h}_j^m-h_ j^m)\quad\text{in}\quad Q_j,& \\
      \varepsilon_j^m(u-v)_+ \quad\text{in}\quad\Omega_T\setminus Q_j. &
    \end{cases}
  \end{displaymath}
  We integrate this estimate, apply Lemma~\ref{lem:epsilon-conv} and
  let $j\to \infty$ to get
  \begin{displaymath}
    \iint_{\Omega_T}(u-v)_+(u^m-v^m)_+\dif x\dif t=0.
  \end{displaymath}
  The claim follows.
\end{proof}

\section{Perron solutions}

The following definition of Perron solutions is based on the comparison principle.

\begin{definition}
  Let $g:\partial_p\Omega_T \to \R$ be given. The \emph{upper class}
  $\UpperClass_g$ consists of the viscosity supersolutions $v$ which
  are locally bounded from below, and satisfy
  \begin{displaymath}
    \liminf_{z\to \xi}v(z)\geq g(\xi)
  \end{displaymath}
  for all $\xi\in \partial_p\Omega_T$. The \emph{upper Perron
    solution} is defined as
  \begin{displaymath}
    \UpperPerron_g(z)=\inf_{v\in \UpperClass_g}v(z).
  \end{displaymath}

  The \emph{lower class} $\LowerClass_g$ consists of all viscosity
  subsolutions $u$ that are locally bounded from above, and satisfy
  \begin{displaymath}
    \limsup_{z\to \xi}u(z)\leq g(\xi)
  \end{displaymath}
  for all $\xi\in \partial_p\Omega_T$. The \emph{lower Perron
    solution} is
  \begin{displaymath}
    \LowerPerron_g(z)=\sup_{u\in \LowerClass_g}u(z).
  \end{displaymath}
\end{definition}

If there exists a function $h\in
C(\overline{\Omega}_T)$ solving the boundary value problem in the
classical sense, then
\begin{displaymath}
  h=\UpperPerron_{g}=\LowerPerron_{g}.
\end{displaymath}
To see this, simply note that the function $h$ belongs to both the
upper class and the lower class.  As we will see, both
$\UpperPerron_g$ and $\LowerPerron_g$ are local weak solutions to the equation.

 A central issue in this theory is the question about when $\LowerPerron_g=\UpperPerron_g$.
If this happens, the boundary function is called \emph{resolutive} and we denote the common function by $\Perron_g$.
An immediate consequence of the comparison principle
(Theorem~\ref{thm:comparison}) is that if $u\in \LowerClass_g$ and
$v\in \UpperClass_g$, then $u\leq v$. Thus
\begin{equation}\label{eq:perron-order}
  \LowerPerron_g\leq \UpperPerron_g
\end{equation}
for bounded boundary functions $g$. 
Our main result (Theorem \ref{thm:resolution}) shows that continuous functions are resolutive.
It should be noticed that even when the solutions coincide, they may
attain wrong boundary values.
If the boundary function $g$ is smooth enough, then 
the weak solutions defined in \eqref{eq:ibvp} and the Perron solutions
coincide, see Theorem \ref{thm:u=H}. 
The main purpose of the definition above is to allow the boundary function $g$ to be general.
In particular, it is not assumed that $g^m\in L^2(0,T;H^1(\Omega))$.

So far, the domain was a space-time cylinder $\Omega_T$. The
definition of upper and lower Perron solutions given above makes sense
in an arbitrary bounded open set $\Upsilon$ in $\R^{N+1}$.  Further,
Lemma~\ref{lem:perron-cont} and Theorem~\ref{thm:perron-weaksol} below
continue to hold, since their proofs are purely local. However, a
comparison principle with the boundary values taken over the whole
topological boundary of $\Upsilon$ is not known for the porous medium
equation. In particular, we do not know whether \eqref{eq:perron-order}
remains true in this generality.

Before addressing the resolutivity question in the next section, we
establish some basic properties of the lower and upper Perron
solutions.

\begin{lemma}\label{lem:perron-cont}
  If $g$ is bounded, then $\UpperPerron_g$ and $\LowerPerron_g$ are
  continuous in $\Omega_T$. 
\end{lemma}
\begin{proof}
  We prove the claim for $\UpperPerron_g$, the other case being
  similar. Take cylinders $U_{t_1,t_2}\Subset
  V_{\sigma_1,\sigma_2}\Subset\Omega_T$, and points $z_1,\, z_2\in
  U_{t_1,t_2}$. Given a positive number $\varepsilon$, we will show
  that
  \begin{displaymath}
    \UpperPerron_g(z_1)-\UpperPerron_g(z_2)<2\varepsilon,
  \end{displaymath}
  provided that $U_{t_1,t_2}$ is sufficiently small. We can find
  functions $v_i^1$ and $v_i^2$ from the upper class such that
  \begin{displaymath}
    \lim_{i\to\infty}v_i^1(z_1)= \UpperPerron_g(z_1)\quad \text{and}\quad
    \lim_{i\to\infty}v_i^2(z_2)= \UpperPerron_g(z_2).
  \end{displaymath}
  Then also $v_i=\min\{v_i^1,v_i^2\}$ is in the upper class, and we have
  \begin{displaymath}
    \lim_{i\to\infty}v_i(z_1)=\UpperPerron_g(z_1) 
    \quad\text{and}\quad
    \lim_{i\to\infty}v_i(z_2)=\UpperPerron_g(z_2) .
  \end{displaymath}
  Let 
  \begin{displaymath}
    w_i=P(v_i,V_{\sigma_1,\sigma_2})\in \UpperClass_g.
  \end{displaymath}
  Then $\UpperPerron_g\leq w_i\leq v_i$, and we have
  \begin{displaymath}
    v_i(z_1)<\UpperPerron_g(z_1)+\varepsilon
    \quad\text{and}\quad
    v_i(z_2)<\UpperPerron_g(z_2)+\varepsilon
  \end{displaymath}
  for sufficiently large $i$. From the above facts and the local
  H\"older continuity of $w_i$, it follows that
  \begin{align*}
    \UpperPerron_g(z_1)-\UpperPerron_g(z_2)\leq & w_i(z_1)-w_i(z_2)+\varepsilon\\
    \leq & \osc_{U_{t_1,t_2}}w_i+\varepsilon\leq 2\varepsilon
  \end{align*}
  by choosing $U_{t_1,t_2}$ in a suitable way. 
  Observe that the  assumption on the boundedness of $g$ implies that the modulus of continuity of $w_i$ is independent of $i$.
  By exchanging the
  roles of $z_1$ and $z_2$, we have
  \begin{displaymath}
    \abs{\UpperPerron_g(z_1)-\UpperPerron_g(z_2)}\leq 2\varepsilon,
  \end{displaymath}
  which completes the proof.
\end{proof}

To prove that Perron solutions are indeed weak solutions to the porous
medium equation, we need some auxiliary results. The first of them is
a Caccioppoli estimate.  The proof can be found in \cite[Lemma~2.15]{KinnunenLindqvist2}.
\begin{lemma}\label{lem:caccioppoli}
  Let $u$ be a weak supersolution  in $\Omega_T$ such that $u^m\in
  L^2(0,T;W^{1,2}(\Omega))$ and $0\leq u\leq M$. Then
  \begin{displaymath}
    \iint_{\Omega_T}\eta^2\abs{\nabla u^m }^2\dif
    x\dif t\leq 16M^{2m}T\int_{\Omega}\abs{\nabla\eta}^2\dif x
    +6M^{m+1}\int_{\Omega}\eta^2\dif x
  \end{displaymath}
  for all nonnegative functions $\eta\in C_0^\infty(\Omega)$.
\end{lemma}
The preceding lemma implies a convergence result in a straightforward
manner, see for example \cite[Proof of Theorem~3.2]{KinnunenLindqvist2}.
\begin{proposition}\label{prop:superconvergence}
  Let $0\leq u_j\leq M$, $j=1,2,\dots$, be weak solutions that
  converge pointwise almost everywhere to a function $u$. Then $u$ is
  also a weak solution.
\end{proposition}

\begin{theorem}\label{thm:perron-weaksol}
  If $g$ is bounded, $\UpperPerron_g$ and $\LowerPerron_g$ are local weak
  solutions in $\Omega_T$.
\end{theorem}
\begin{proof}
  We give the proof for $\UpperPerron_g$, the case of $\LowerPerron_g$
  being again symmetrical.  Let $q_n$, $n=1,2,3,\ldots$, be an
  enumeration of the points in $\Omega_T$ with rational
  coordinates. The first aim is to construct functions in the upper
  class converging to $\UpperPerron_g$ at the points $q_n$. To
  accomplish this, let $v^n_i\in\UpperClass_g$ be such that
  \begin{displaymath}
    \UpperPerron_g(q_n)\leq v_i^n(q_n)<\UpperPerron_g(q_n)+\frac1i, \quad i=1,2,3,\ldots, 
  \end{displaymath}
  and define
  \begin{displaymath}
    w_i=\min\{v_1^1,v_2^1, \ldots, v_i^1,v_1^2,v_2^2, \ldots,
    v_i^2,\ldots,v_1^i,v_2^i,\ldots, v_i^i \}.
  \end{displaymath}
  Then $w_i\in\UpperClass_g$, $w_1\geq w_2\geq w_3 \ldots$, and 
  \begin{displaymath}
    \UpperPerron(q_n)\leq w_i(q_n)\leq v_i^n(q_n), 
      \end{displaymath}
  for $i\geq n$.
It follows that 
  \begin{displaymath}
    \lim_{i\to \infty} w_i(q_n)=\UpperPerron(q_n)
  \end{displaymath}
  at each point $q_n$.  Let $U_{t_1,t_2}\Subset \Omega_T$ be an
  arbitrary regular cylinder and denote
  \begin{displaymath}
    W_i=P(w_i,U_{t_1,t_2}).
  \end{displaymath}
  Then $\UpperPerron_g\leq W_i\leq w_i$, the sequence $(W_i)$ is
  decreasing, and its limit $W$ is a solution in $U_{t_1,t_2}$ by
  Proposition~\ref{prop:superconvergence}. 
 At every point $q_n$ we have
  \begin{displaymath}
    W(q_n)=\lim_{i\to\infty}W_i(q_n)=\UpperPerron(q_n).
  \end{displaymath}
  Both $W$ and $\UpperPerron_g$ are continuous in $U_{t_1,t_2}$,
  and they coincide on a dense subset; hence they must coincide
  everywhere. Since $W$ is a solution in $U_{t_1,t_2}$, so is
  $\UpperPerron_g$. The property of being a solution is local, so the
  proof is complete.
\end{proof}

\section{Resolutivity}

The following theorem is our main result.  It states that
\emph{continuous functions are resolutive}.

\begin{theorem}[Resolutivity]\label{thm:resolution}
  If $g:\partial_p\Omega_T\to \R$ is continuous, then
  \begin{displaymath}
    \UpperPerron_g=\LowerPerron_g.
  \end{displaymath}
\end{theorem}

To prove the resolutivity theorem, by approximation we first reduce
the situation to smooth boundary values. 
For smooth boundary values, we need to construct functions belonging to the upper class $\UpperClass_g$ that are sufficiently smooth in time and attain the correct boundary and initial values in the weak sense.
We do this by solving a penalized equation. For this purpose, assume the function $g$ to be
continuously differentiable in $\overline{\Omega}_T $ and such that
$g^m\in C^2(\overline{\Omega}_T)$.  Then
\begin{displaymath}
  \Psi=\frac{\partial g}{\partial t}-\Delta g^m
\end{displaymath}
is bounded.  We will use the positive part $\Psi_+=\max\{\Psi,0\}$
below. Choose a number $\delta>0$, and let $\zeta_\delta:\R\to \R$ be
a Lipschitz function such that $0\leq \zeta_\delta(s)\leq 1$,
$\zeta_\delta(s)=1$ for all $s\geq 0$, $\zeta_\delta(s)=0$ for all
$x\leq -\delta$, and $\abs{ \zeta'_\delta(s)}\leq 2/\delta$. We have
the following existence result, see~\cite{BCL}.

\begin{proposition}\label{prop:penalized-pme}
  Let $g$ be continuously differentiable in $\overline{\Omega}_T $ and
  such that $g^m\in C^2(\overline\Omega_T)$.
  Then there exists a bounded weak solution $u$ such that $u^m\in
  L^2(0,T;H^1(\Omega))$ to the boundary value problem
  \begin{displaymath}
    \begin{cases}
         \dfrac{\partial u}{\partial t}-\Delta u^m=\zeta_\delta(g^m-u^m)\Psi_+\quad\text{in}\quad\Omega_T,&\\
      u^m-g^m\in L^2(0,T;H^1_0(\Omega)), & \\
      u(x,0)=g(x,0),
    \end{cases}
  \end{displaymath}
  satisfying the inequality $u\geq g$ in $\Omega_T$.
\end{proposition}

\begin{remark}\label{rem:poweapprox}
  In the proof of Theorem~\ref{thm:resolution} below, we need to
  choose approximations of a given continuous function $g$ so that the
  smoothness assumptions of Proposition~\ref{prop:penalized-pme}
  hold. This is accomplished by approximating a suitable smaller power
  $g^\alpha$, $\alpha\leq 1$, of the function rather than the function
  $g$ itself. Indeed, we may express the derivatives of the powers one
  and $m$ in terms of the derivatives of the power $\alpha$.  Some
  simple calculations show that the choice $\alpha=\min\{1,m/2\}$ will
  do.
\end{remark}

\begin{remark}\label{rem:subsols}
  Due to our assumption that the boundary values are positive, the
  roles of upper and lower solutions in the proof of
  Theorem~\ref{thm:resolution} are not quite symmetric. For
  subsolutions, we need a version of
  Proposition~\ref{prop:penalized-pme} where the solutions can change
  sign. See pp.~97-100 in \cite{VazquezBook} for the modifications
  needed to the arguments in \cite{BCL}.
\end{remark}

We need an energy estimate for the time derivative of a solution to
the above equation. For similar results, see
\cite[Proposition~13]{JLVNotes} and \cite[Section~3.2.5]{VazquezBook}.

\begin{theorem}\label{thm:obstacle-timederiv}
  Assume that $f\in L^\infty(\Omega_T)$. Let $u$ be a bounded weak solution to the boundary value problem
  \begin{displaymath}
    \begin{cases}
         \dfrac{\partial u}{\partial t}-\Delta u^m=f \quad\text{in}\quad\Omega_T, &\\
      u^m-g^m\in L^2(0,T;H^1_0(\Omega)),& \\
      u(x,0)=g(x,0).
    \end{cases}
  \end{displaymath}
  Then 
  \begin{displaymath}
    \frac{\partial u^{(m+1)/2}}{\partial t}\in L^2(\Omega_T),
  \end{displaymath}
  with the estimate
     \begin{align*}
    \iint_{\Omega_T}&\scabs{\frac{\partial u^{(m+1)/2}}{\partial t}}^2\dif x\dif
    t+\int_{\Omega}\abs{\nabla (u^m-g^m)}^2(x,T)\dif x\\
    &\leq  c\bigg(
      \int_{\Omega}\left|(g^m)_t(x,T)u(x,T)-(g^m)_t(x,0)g(x,0)\right|
      \dif x\\&\quad +
      \iint_{\Omega_T}\left(\abs{u}^2+u^{m-1}(\abs{f}^2+\abs{\Delta
          g^m}^2)\right)
      \dif x\dif
      t\\&\quad
      +\iint_{\Omega_T}\left(\abs{(g^m)_t}^2+\abs{\Delta g^m}^2+\abs{(g^m)_{tt}}^2
        +\abs{f}^2\right)\dif x\dif
      t\bigg).
    \end{align*}
   Furthermore, we have 
   \[
   \frac{\partial u^{q}}{\partial t}\in L^2(\Omega_T)
   \] 
   for any $q\geq (m+1)/2$.
\end{theorem}
\begin{proof}
  First we assume that $u$ is smooth in $\overline\Omega_T$. This assumption may be removed
  by a standard approximation argument, see \cite[Proof of Theorem~5.5]{VazquezBook} and \cite[Section~1.3.2]{Kiinalaiset}. Denote
  $w=u^{(m+1)/2}$. Then
  \begin{align*}
    \scabs{\frac{\partial w}{\partial t}}^2&=
    \frac{(m+1)^2}{4}u^{m-1}\abs{u_t}^2\\
    &=\frac{(m+1)^2}{4m}(u^m)_tu_t
    = \frac{(m+1)^2}{4m}\big((u^m-g^m)_tu_t+(g^m)_tu_t\big)\\
    &= \frac{(m+1)^2}{4m}\big((u^m-g^m)_t(\Delta u^m-\Delta g^m)\\&\quad+(g^m)_tu_t+\Delta
    g^m(u^m-g^m)_t
    +f(u^m-g^m)_t\big).
  \end{align*}
  We focus on the first term after the last equality. To this
  end, we note that
  \begin{align*}
    \frac{1}{2}\frac{d}{dt}\int_{\Omega}\abs{\nabla u^m-g^m}^2\dif x=&
    \int_{\Omega}\nabla (u^m-g^m)\cdot\nabla (u^m-g^m)_t\dif x\\
    =&-\int_{\Omega}\Delta(u^m-g^m)(u^m-g^m)_t\dif x,
  \end{align*}
  since $u^m-g^m$ has zero boundary values on the lateral boundary. Thus an integration gives
 \[
  \begin{split}
    &\iint_{\Omega_T}\scabs{\frac{\partial w}{\partial t}}^2\dif x\dif t
      =-\frac{(m+1)^2}{8m}\int_0^T\frac{d}{dt}\int_{\Omega}\abs{\nabla(u^m-g^m)}^2\dif
      x\dif t\\
      &\quad+\frac{(m+1)^2}{4m} \iint_{\Omega_T}\big((g^m)_tu_t+(u^m-g^m)_t\Delta g^m+f(u^m-g^m)_t\big)\dif x\dif t\\
      &=-\frac{(m+1)^2}{8m}\left.\int_{\Omega}\abs{\nabla(u^m-g^m)}^2\dif
        x\right|_{0}^T\\
      &\quad+\frac{(m+1)^2}{4m}\left(\iint_{\Omega_T}-(g^m)_{tt} u\dif x\dif t
      +\left.\int_{\Omega}(g^m)_t u\dif x\right|_0^T\right)\\
      &\quad+\frac{(m+1)^2}{4m}\iint_{\Omega_T}\big((u^m)_t\Delta g^m-(g^m)_t\Delta
      g^m-f(g^m)_t+(u^m)_tf\big)\dif x\dif t.
    \end{split}
  \]
  To proceed, we compute 
  \begin{displaymath}
    \frac{\partial u^m}{\partial t}=u^{m-(m+1)/2}\frac{\partial
      w}{\partial t},
  \end{displaymath}
  and apply Young's inequality to the two terms containing the time
  derivative of $u^m$ to get
  \begin{displaymath}
    \iint_{\Omega_T}\abs{(u^m)_t \Delta g^m}\dif x\dif t\leq \varepsilon
    \iint_{\Omega_T}\scabs{\frac{\partial
      w}{\partial t}}^2\dif x\dif t
  +c_\varepsilon\iint_{\Omega_T}u^{m-1}\abs{\Delta g^m}^2\dif x\dif t
  \end{displaymath}
  and
  \begin{displaymath}
    \iint_{\Omega_T}\abs{(u^m)_t f}\dif x\dif t\leq
    \varepsilon\iint_{\Omega_T} \scabs{\frac{\partial w}{\partial
        t}}^2\dif x\dif t
    +c_{\varepsilon}\iint_{\Omega_T}u^{m-1}\abs{f}^2\dif x\dif t.
  \end{displaymath}
  We insert these inequalities into the estimate above, choose a sufficiently
  small $\varepsilon$, and absorb the matching terms to get
  \begin{displaymath}
      \begin{aligned}
    \iint_{\Omega_T}&\scabs{\frac{\partial w}{\partial t}}^2\dif x\dif
    t+\int_{\Omega}\abs{\nabla (u^m-g^m)}^2(x,T)\dif x \\
    &\leq  c\bigg(\int_{\Omega}\abs{\nabla(u^m-g^m)}^2(x,0)\dif x+
      \iint_{\Omega_T}\abs{(g_{tt})^mu}\dif x\dif
      t\\&\quad+\int_{\Omega}\abs{(g^m)_t(x,T)u(x,T)-(g^m)_t(x,0)u(x,0)}\dif x\\
      &\quad+\iint_{\Omega_T}u^{m-1}\abs{\Delta g^m}^2\dif x\dif t
      +\iint_{\Omega_T}\abs{(g^m)_t\Delta g^m}\dif x\dif t\\
      &\quad+\iint_{\Omega_T}\abs{f(g^m)_t}\dif x\dif t+\iint_{\Omega_T}u^{m-1}\abs{f}^2\dif x\dif t\bigg).
    \end{aligned}
  \end{displaymath}
  We recall that $u^m=g^m$ at the initial time, so the required
  estimate follows from an application of Cauchy's inequality.
\end{proof}

\begin{lemma}\label{lem:obst-estimate}
  Let $g$ satisfy the smoothness assumptions of
  Proposition~\ref{prop:penalized-pme},
  and let $v$ be the solution to the boundary value problem of
  Proposition~\ref{prop:penalized-pme}. Let $u=P(v,D_{t_1,T})$ be the
  Poisson modification of $v$ with respect to a regular space-time
  cylinder $D_{t_1,T}$ with $D\Subset\Omega$ and $0<t_1<T$. Then
  \begin{align*}
    \iint_{D_{t_1,T}}&\abs{\nabla u^m}^2\dif x\dif
    t+\sup_{t_1<t<T}\int_{D}u^{m+1}(x,t)\dif x \\
    &\leq  c\bigg(\iint_{D_{t_1,T}}\bigg(\abs{\nabla
        v^m}^2+\abs{v}^2 +\scabs{\frac{\partial v^m}{\partial t}}^2\bigg)\dif x\dif t\\
    &\quad+\sup_{t_1<t<T}\int_{D}v^{m+1}(x,t)\dif x\bigg).
  \end{align*}
\end{lemma}
\begin{proof}
 The formal computations below are
  justified rigorously by a standard application of a suitable
  mollification in the time direction. Since $u$ is a solution in
  $D_{t_1,T}$ with boundary values given by $v^m$, we may use
  $u^m-v^m$ as a test function and have
  \begin{equation}\label{eq:energy-est1}
    \iint_{D_{t_1,T}}\left(\frac{\partial u}{\partial t}(u^m-v^m)
    +\nabla u^m\cdot\nabla (u^m-v^m)\right)\dif x\dif t=0.
  \end{equation}
  The next goal is to eliminate the time derivative of $u$. We use the fact
  that $(u^{m+1})_t=(m+1)u_t u^m$ and integrate by parts in the other
  term to get
  \begin{align*}
    \iint_{D_{t_1,T}}&\frac{\partial u}{\partial t}(u^m-v^m)\dif x\dif t\\
      &=  \frac{1}{m+1}\left(\int_{D}u^{m+1}(x,T)\dif
        x-\int_{D}u^{m+1}(x,t_1)\dif x \right)\\&\quad
      +\iint_{D_{t_1,T}}u\frac{\partial v^m}{\partial t}\dif x\dif
      t\\
      &\quad-\int_{D}u(x,T)v^m(x,T)\dif x+\int_{D}u(x,t_1)v^m(x,t_1)\dif x.
        \end{align*}
  This leads to the estimate
  \begin{align*}
    0&\leq \iint_{D_{t_1,T}}\abs{\nabla u^m}^2\dif x\dif
    t+\frac{1}{m+1}\int_{D}u^{m+1}(x,T)\dif x\\
    &\leq  \frac{1}{m+1}\int_{D}v^{m+1}(x,t_1)\dif x + \int_{D}v^{m+1}(x,T)\dif x
    -\iint_{D_{t_1,T}}u\frac{\partial v^m}{\partial
      t}\dif x\dif t\\&\quad+\iint_{D_{t_1,T}}\nabla u^m\cdot\nabla v^m\dif x\dif t,
  \end{align*}
  since $u\leq v$. By Young's inequality, we obtain
  \begin{align*}
    \iint_{D_{t_1,T}}&\nabla u^m\cdot\nabla v^m\dif x\dif t\\
    &\leq\varepsilon\iint_{D_{t_1,T}}\abs{\nabla u^m}^2\dif x\dif t+
    c_\varepsilon\iint_{D_{t_1,T}}\abs{\nabla v^m}^2\dif x\dif t.
  \end{align*}
  We insert this into the previous estimate, and absorb the matching
  terms. We arrive at
 \begin{align*}
    \int_{D_{t_1,T}}&\abs{\nabla u^m}^2\dif x\dif
    t+\frac{1}{m+1}\int_{D}u^{m+1}(x,T)\dif x\\
      &\leq c \bigg(\int_{D}v^{m+1}(x,t_1)\dif x
        +\int_{D}v^{m+1}(x,T)\dif x\\ 
        &\quad+\int_{D_{t_1,T}}\abs{v}\scabs{\frac{\partial v^m}{\partial
            t}}+\iint_{D_{t_1,T}}\abs{\nabla v^m}^2\dif x\dif t\bigg).
    \end{align*}
  The proof is then completed by estimating the term with
  $v^{m+1}(x,T)$ by a supremum over time, and by replacing $T$ by
  $t_1<\tau<T$ such that
  \begin{displaymath}
    \int_{D}v^{m+1}(x,\tau)\dif x\geq
   \sup_{t_1<t<T} \frac{1}{2}\int_{D}v^{m+1}(x,t)\dif x,
  \end{displaymath}
  and applying Young's inequality.
\end{proof}

\begin{proof}[Proof of Theorem~\ref{thm:resolution}]
  By extension, we may assume that $g$ is defined in the whole of
  $\R^{N+1}$.  We first show that it suffices to prove that for smooth
  boundary values $g$ that both the upper and lower Perron solution
  agree with the unique weak solution solution with boundary and
  initial values $g$, in the sense of \eqref{eq:ibvp}.  Let us denote
  $\varepsilon_j=1/j$, $j=1,2,\dots$. There exist functions
  $\varphi_j$ satisfying the smoothness assumptions of
  Proposition~\ref{prop:penalized-pme}, see
  Remark~\ref{rem:poweapprox}, converging uniformly to $g$ and such
  that
  \begin{displaymath}
    \varphi_j^m\leq g^m \leq \varphi_j^m+\varepsilon_j^m.
  \end{displaymath}
  Assuming the above conclusion for smooth functions, we get
  \begin{displaymath}
    \Perron_{\varphi_j}\leq \LowerPerron_g\leq \UpperPerron_g\leq 
    \Perron_{(\varphi_j^m+\varepsilon_j^m)^{1/m}}.
  \end{displaymath}
 Since
  \[
  \abs{\Perron_{\varphi_j}-\Perron_{(\varphi_j^m+\varepsilon_j^m)^{1/m}}}\to 0
  \] 
  as $j\to \infty$ by Lemma~\ref{lem:epsilon-conv}, it follows that $\LowerPerron_g=\UpperPerron_g$ almost everywhere.
  The conclusion that $\LowerPerron_g=\UpperPerron_g$ everywhere follows by continuity of the Perron solutions.

  Let us then assume that $g$ is smooth, and let $h$ be the unique
  weak solution with initial and boundary values given by $g$, i.e.
  $h^m-g^m\in L^2(0,T; H^1_0(\Omega))$ and \eqref{eq:ibvp} holds. We
  need to show that $h\geq \UpperPerron_g$; the problem is that we do
  not know whether $h$ belongs to the upper class or not.  To deal
  with this, let $v$ be the solution of the the penalized boundary
  value problem of Proposition~\ref{prop:penalized-pme}. Then also
  \[
  v^m-g^m \in L^2(0,T; H^1_0(\Omega)).
  \]  
  Exhaust $\Omega_T$ by an
  increasing sequence of regular cylinders $Q_j=U_j\times (t_j,T)$, and let
  $w_j=P(v,Q_j)$, $j=1,2,\dots$. Then $w_j\in \UpperClass_g$, the
  sequence $(w_j)$ is decreasing, and $\UpperPerron_g\leq w_j$. The
  limit function
  \[
  w=\lim_{j\to \infty}w_j
  \] 
  is a solution in $\Omega_T$, and 
  \begin{displaymath}
    w\geq \UpperPerron_g
  \end{displaymath}
  since $w$ is a pointwise limit of functions in the upper class. It
  remains to show that $w$ has the boundary and initial values given
  by $g$, since then by the uniqueness of weak solutions, we have
  \[
  h=w\geq \UpperPerron_g.
  \] 
  
  To check the lateral boundary values, note that the sequence $(w_
  j^m-g^m)$ is bounded in $L^2(0,T;H^1_0(\Omega))$ by
  Lemma~\ref{lem:obst-estimate}. It follows that 
  \[
  w^m-g^m \in L^2(0,T; H^1_0(\Omega)),
  \] since $L^2(0,T; H^1_0(\Omega))$ is a closed
  subspace of $L^2(0,T; H^1(\Omega))$ and weak limits must agree with
  pointwise limits.

  We use the criterion \eqref{eq:initial-vals-dist} to show that the
  initial values of the limit function $w$ are given by the function
  $g(x,0)$. Let $\eta\in C^\infty_0(\Omega)$ be arbitrary. Choose a
  time instant $0<t<T$, and $j$ large enough, so that $t_j<t$ and 
  so that the support of $\eta$ is contained in $U_j$. 
  We have
  \[
  \begin{split}
    &\scabs{\int_{\Omega}(w(x,t)-g(x,0))\eta(x)\dif x}
    \leq  \scabs{\int_{\Omega}(w(x,t)-w_j(x,t))\eta(x)\dif x}\\
    &\quad+\scabs{\int_{\Omega}(w_j(x,t)-w_j(x,t_j))\eta(x)\dif x}
    +\scabs{\int_{\Omega}(v(x,t_j)-g(x,0))\eta(x)\dif x}
  \end{split}
  \]
  by adding and substracting suitable terms, using the triangle
  inequality, and the fact that $w_j(x,t_j)=v(x,t_j)$ on the support
  of $\eta$. The first and third terms on the right tend to zero as
  $j\to \infty$. To deal with the second term, we formally test the
  equation satisfied by $w_j$ with $\varphi=\eta\chi_{(t_j,t)}$, where
  $\chi_{(t_j,t)}$ is the characteristic function of the interval
  $(t_j,t)$.  This can be justified by an approximation argument. We
  get
  \begin{equation*}
    \scabs{\int_{\Omega}(w_j(x,t)-w_j(x,t_j))\eta\dif
      x}=\scabs{\iint_{U_j\times(t_j,t)}\nabla w_j^m\cdot \nabla\eta\dif x\dif t}.
  \end{equation*}
  We estimate the right hand side by H\"older's inequality to get
  \begin{displaymath}
    \scabs{\iint_{U_j\times(t_j,t)}\nabla w_j^m\cdot \nabla\eta\dif
      x\dif t}\leq \abs{\Omega}^{1/2}\abs{t-t_j}^{1/2}\norm{\nabla
      w_j^m}_{L^2(\Omega_T)}\norm{\nabla \eta}_{L^{\infty}(\Omega)},
  \end{displaymath}
  where we also used the fact that $\abs{U_j\times(t_j,t)}\leq
  \abs{\Omega}\abs{t-t_j}$. Since the norm of $\nabla w_j^m$ can be
  controlled independently of $j$ by applying
  Lemma~\ref{lem:obst-estimate}, we may use this estimate for the
  second term to get
  \begin{displaymath}
    \scabs{\int_{\Omega}(w(x,t)-g(x,0))\eta(x)\dif x}\leq c
    t\norm{\nabla \eta}_{L^\infty(\Omega)}
  \end{displaymath}
  after letting $j\to \infty$.  Since $\eta$ was arbitrary, letting
  $t\to 0$ shows that \eqref{eq:initial-vals-dist} holds for the
  function $w$, as desired.

  By a similar argument using the variant of
  Proposition~\ref{prop:penalized-pme} described in Remark
  \ref{rem:subsols}, we see that $h\leq \LowerPerron_g$, so that
  \begin{displaymath}
    h\leq \LowerPerron_g\leq \UpperPerron_g\leq h,
  \end{displaymath}
  which completes the proof.
\end{proof}

The second part of the previous proof gives the following uniqueness result.

\begin{theorem}\label{thm:u=H}
  Let $g$ satisfy the smoothness assumptions of
  Proposition~\ref{prop:penalized-pme}
  and let $u$ be the weak solution to the boundary value problem in the sense of \eqref{eq:ibvp}. 
Then $u=\Perron_g$.
\end{theorem}

\def\cprime{$'$}


\begin{thebibliography}{10}

\bibitem{AL}
B.~Avelin and T.~Lukkari.
\newblock Lower semicontinuity of weak supersolutions to the porous medium
  equation.
\newblock Preprint. Available at \url{http://arxiv.org/abs/1312.2820}.

\bibitem{Barenblatt}
G.~I. Barenblatt.
\newblock On self-similar motions of a compressible fluid in a porous medium.
\newblock {\em Akad. Nauk SSSR. Prikl. Mat. Meh.}, 16:679--698, 1952.

\bibitem{BrandleVazquez}
C. Br\"andle and J. L. V\' azquez.
Viscosity solutions for quasilinear degenerate parabolic equations of porous medium type.
{\em Indiana Univ. Math. J.}, 54(3):817--860, 2005. 

\bibitem{BCL}
V.~B\"ogelein, T.~Lukkari, and C.~Scheven.
\newblock The obstacle problem for the porous medium equation.
\newblock Submitted. Available at \url{https://www.mittag-leffler.se/preprints/files/IML-1314f-33.pdf}.

\bibitem{CV}
L. Caffarelli and J. L. V\' azquez.
Viscosity solutions for the porous medium equation. Differential equations
{\em Proc. Sympos. Pure Math.}, 65:13--26, Amer. Math. Soc., Providence, RI, 1999. 

\bibitem{DahlbergKenig}
B.~E.~J. Dahlberg and C.~E. Kenig.
\newblock Nonnegative solutions of the porous medium equation.
\newblock {\em Comm. Partial Differential Equations}, 9(5):409--437, 1984.

\bibitem{DK}
P.~Daskalopoulos and C.~E. Kenig.
\newblock {\em Degenerate diffusions -- Initial value problems and local
  regularity theory}, volume~1 of {\em EMS Tracts in Mathematics}.
\newblock European Mathematical Society (EMS), Z\"urich, 2007.

\bibitem{DiBenedettoFriedman}
E.~DiBenedetto and A.~Friedman.
\newblock H\"older estimates for nonlinear degenerate parabolic systems.
\newblock {\em J. Reine Angew. Math.}, 357:1--22, 1985.

\bibitem{KinnunenLindqvist2}
J.~Kinnunen and P.~Lindqvist.
\newblock Definition and properties of supersolutions to the porous medium
  equation.
\newblock {\em J. Reine Angew. Math.}, 618:135--168, 2008.

\bibitem{Perron}
O.~Perron.
\newblock Eine neue {B}ehandlung der ersten {R}andwertaufgabe f\"ur {$\Delta
  u=0$}.
\newblock {\em Math. Z.}, 18(1):42--54, 1923.

\bibitem{Riesz}
F.~Riesz.
\newblock Sur les {F}onctions {S}ubharmoniques et {L}eur {R}apport \`a la
  {T}h\'eorie du {P}otentiel.
\newblock {\em Acta Math.}, 48(3-4):329--343, 1926.

\bibitem{JLVNotes}
J.~L. V{\'a}zquez.
\newblock An introduction to the mathematical theory of the porous medium
  equation.
\newblock In {\em Shape optimization and free boundaries ({M}ontreal, {PQ},
  1990)}, volume 380 of {\em NATO Adv. Sci. Inst. Ser. C Math. Phys. Sci.},
  pages 347--389. Kluwer Acad. Publ., Dordrecht, 1992.

\bibitem{VazquezBook}
J.~L. V{\'a}zquez.
\newblock {\em The {P}orous {M}edium {E}quation -- Mathematical theory}.
\newblock Oxford Mathematical Monographs. The Clarendon Press Oxford University
  Press, Oxford, 2007.

\bibitem{Watson}
N.~A. Watson.
\newblock {\em Introduction to heat potential theory}, volume 182 of {\em
  Mathematical Surveys and Monographs}.
\newblock American Mathematical Society, Providence, RI, 2012.

\bibitem{WienerResol}
N.~Wiener.
\newblock Note on a paper of {O}. {P}erron.
\newblock {\em J. Math. Phys.}, 4:21--32, 1925.

\bibitem{Kiinalaiset}
Z.~Wu, J.~Zhao, J.~Yin, and H.~Li.
\newblock {\em Nonlinear diffusion equations}.
\newblock World Scientific Publishing Co. Inc., River Edge, NJ, 2001.
\newblock Translated from the 1996 Chinese original and revised by the authors.

\bibitem{ZeldovichKompaneets}
Ya.~B. Zel{\cprime}dovi{\v{c}} and A.~S. Kompaneec.
\newblock On the theory of propagation of heat with the heat conductivity
  depending upon the temperature.
\newblock In {\em Collection in honor of the seventieth birthday of academician
  {A}. {F}. {I}offe}, pages 61--71. Izdat. Akad. Nauk SSSR, Moscow, 1950.

\end{thebibliography}
\end{document}